\documentclass{amsart}

\usepackage{amsmath, amsfonts, amssymb, amsaddr, setspace, xypic}
\input xy
\xyoption{all}
\newtheorem{theorem}{Theorem}
\newtheorem{lemma}[theorem]{Lemma}
\newtheorem{lemma-def}[theorem]{Lemma-Definition}
\newtheorem{proposition}[theorem]{Proposition}
\newtheorem{corol}[theorem]{Corollary}
\newtheorem{definition}[theorem]{Definition}

\newcommand{\Q}{\mathbb{Q}}

\newcommand{\Z}{\mathbb{Z}}

\newcommand{\N}{\mathbb{N}}

\newcommand{\R}{\mathbb{R}}

\newcommand{\ord}{\mathrm{ord}\,}

\newcommand{\vu}{\nu}

\DeclareMathOperator{\sq}{\square}

\title[Parametric rectilinearization and rationality of $p$-adic integrals]
{Parametric rectilinearization for $p$-adic semi-algebraic sets \\and rationality of $p$-adic integrals}\author{Eva Leenknegt}
\address{Department of Mathematics, \\Purdue University, 
\\West Lafayette, Indiana 47907, USA}
\email{eleenkne@math.purdue.edu}
\date{}
\begin{document}
\maketitle

\begin{abstract}
We present a rectilinearization theorem for $p$-adic semi-algebraic sets depending on parameters.
As an application of our main theorem we present an alternative proof of a rationality result for parametric $p$-adic integrals, due to Denef.
\end{abstract}
\section{Introduction}

The technique of rectilinearization was originally used by Cluckers \cite{clu-2000}, to show that any two infinite $p$-adic semi-algebraic sets are isomorphic if they have the same dimension.
To achieve this result, an important step was to show that any semi-algebraic $p$-adic set was isomorphic to a finite number of sets of a specific, simple form. It was already noted by Cluckers at that time that the result might be useful to $p$-adic integration, especially since the used isomorphisms also appeared to be very basic functions.

In a previous paper \cite{clu-lee-2008}, we presented such an extended rectilinearization result, which states that 
any
semialgebraic $p$-adic set can be partitioned into finitely many
parts each of which is semi-algebraically isomorphic to a
Cartesian power of basic subsets $\Z_p^{(k)}$ of $\Q_p$, where
$\Z_p^{(k)}$ is the set of $p$-adic integers (of any nonnegative
order) having coefficients $1,0,\ldots,0$ in their $p$-adic expansion,
with $k-1$ zeros, see (\ref{(k)}) below. Moreover, we could also ensure that  the order of a finite number of given semi-algebraic fuctions, and the order of the Jacobian of the occuring isomorphisms, would equal the order of a monomial with integer powers. We used this rectilinearization result to give a new, very simple proof of the rationality of certain $p$-adic integrals. Rationality of such integrals and the related Poincare series had originally been proved by Denef \cite{denef-84} in two different ways, one of which was based on cell decomposition techniques.

This paper,  which was inspired by a similar result of  Cluckers for Presburger sets \cite{clu-presb03}, presents a parametric version of the rectilinearization theorem mentioned above. Our motivation here is similar to that in the previous paper, namely that this result allows us to give alternative, simple proofs of the rationality of parametric $p$-adic integrals. As an application of our main theorem we present an alternative proof of a rationality result for parametric $p$-adic integrals, due to Denef  \cite{denef-85, denef-2000}.

\subsection{Notation and terminology}
Let $p$ denote a fixed prime number, $\Q_p$ the field of $p$-adic
numbers and $K$ a fixed finite field extension of $\Q_p$.  For $x\in
K$ let $\ord(x)\in\Z\cup\{+\infty\}$ denote the valuation of $x$. Let
$R=\{x\in K\mid \ord(x)\geq0\}$ be the valuation ring, and let $q_K$
denote the cardinality of the residue field of $K$. Put
$K^\times=K\setminus\{0\}$ and for $n\in\N_0$ let $P_n$ be the set
$\{x\in K^\times\mid\exists y\in K\ : y^n=x\}$. We call a subset of
$K^n$ {\em semi-algebraic } if it is a Boolean combination (i.e.
obtained by taking finite unions, complements and finite
intersections) of sets of the form $\{x\in K^m\mid f(x)\in P_n\}$,
with $f(x)\in K[X_1,\ldots,X_m]$. The quantifier elimination result by Macintyre \cite{mac-76} implies that the collection of semi-algebraic
sets is closed under taking projections $K^m\to K^{m-1}$. Moreover, also sets of the form $\{x\in K^m \ \mid \ \ord(f(x))\leq
\ord(g(x))\}$ with $f(x),g(x)\in K[X_1,\ldots,X_m]$ are semi-algebraic
(see Denef's alternative proof \cite{denef-86} of Macintyre's theorem). 
A function $f:A\to B$ is called semi-algebraic
if its graph is a semi-algebraic set; if such a function $f$ is a bijection,
we call $f$
 an {\em isomorphism}. 
 By a \emph{finite partition} of a
semi-algebraic set we mean a partition into finitely many
semi-algebraic sets.
\par
 Let $\pi$ be a fixed element of $R$ with
$v(\pi)=1$, thus $\pi$ is a uniformizing parameter for $R$. 
For a
semi-algebraic set $X\subset K$ and $k>0$ we write
\begin{equation}\label{(k)}
X^{(k)}=\{x\in X\ \mid\ x\not=0 \mbox{ and } \ord(\pi^{-\ord(x)}x-1)\geq
k\},
\end{equation}
which is semi-algebraic (see \cite{denef-86}, Lemma 2.1); $X^{(k)}$
consists of those points $x\in X$ which have a $p$-adic expansion
$x=\sum_{i=s}^\infty a_i\pi^i$ with $a_s=1$ and $a_i=0$ for
$i=s+1,\ldots,s+k-1$.  \\

\noindent We recall a form of Hensel's Lemma and a corollary.
\begin{lemma}[Hensel]\label{Hensel}
Let $f(t)$ be a polynomial over $R$ in one variable $t$, and let
$\alpha\in R,\ e\in\N$. Suppose that $\ord( f(\alpha))>2e$ and $
  \ord(f'(\alpha))\leqslant e$, where $f'$ denotes the derivative of $f$. Then there
exists a unique $\bar\alpha\in R$ such that $f(\bar\alpha)=0$ and
$\ord(\bar\alpha-\alpha)>e$.
\end{lemma}
\begin{corol}
\label{corhensel}
Let $n>1$ be a natural number. For each $k> \ord(n)$, and $k'=k+ \ord(n)$
the function $ K^{(k)}  \to P_n^{(k')}:x\mapsto x^n $ is an
isomorphism.
\end{corol}
\par
\noindent We will also need the following cell decomposition theorem, a proof of which can for example be found in \cite{clu-2000}. Prior to that, this result was used in the proofs of \cite{denef-84}.
\begin{lemma}\label{partition}
Let $X\subset K^m$ be semi-algebraic and $b_j:K^m\to K$
semi-algebraic functions for $j=1,\ldots,r$. Then there exists a
finite partition of $X$ s.t. each part $A$ has the form
\begin{eqnarray*}
A & = & \{x\in K^m \mid \hat x\in D, \  \ord a_1(\hat
    x)\sq_1 \ord(x_m-c(\hat x))\sq_2 \ord a_2(\hat x),\\
 & & \qquad\qquad\qquad x_m-c(\hat x)\in \lambda P_n\},\end{eqnarray*}
and such that for each $x\in A$ we have that
\[ \ord(b_j(x))=\frac{1}{n} \ord((x_m-c(\hat x))^{\mu_j}d_j(\hat x)),\]
where $\hat x=(x_1,\ldots, x_{m-1})$, $D\subset K^{m-1}$
is semi-algebraic, the set $A$ projects surjectively onto $D$, $n>0$,
$\mu_j\in\Z$, $\lambda\in K$, $c,d_j:K^{m-1} \to K$ and
$a_i:K^{m-1}\to K^\times$ are semi-algebraic functions, and each $\sq_i$ is
either
$\leq$ or no condition. If $\lambda = 0$, we use the conventions that  $\mu_j =0$ and $0^0 =1$, and thus $\ord(b_j(x)) = \frac1{n}\,\ord(d_j(\hat{x}))$.\\
Moreover, $D$ has the structure of a $K$-analytic manifold and the
functions $c$, $a_i$ and $d_j$ are $K$-analytic on $D$.
\end{lemma}

\section{A parametric version of rectilinearization}

\noindent Write $\pi_m: K^{m+r} \to K^r:(x_1, \ldots, x_{m+r}) \mapsto (x_1, \ldots, x_m)$ for the projection map onto the first $m$ coordinates. Given a set $A\subseteq K^{m+n}$ and $\xi \in \pi_m(A)$, we use the notation $A_{\xi}$ to denote the projection
\[A_{\xi}:= \{x\in K^n \ | \ (\xi,x)\in A \}.\]
We say that a semi-algebraic set $S \subseteq K^n$ is bounded if there exists a tuple $a \in (K^{\times})^n$, such that for each $x\in S$ and $i=1, \ldots, n$, we have that
\[-\ord a_i \leqslant \ord x_i \leqslant \ord a_i.\]
The following definitions are central.

\begin{definition}\textup{
Let $A \subseteq K^{r+m}$ be a semi-algebraic set.\\ We say that a family of functions $f =\{f_{\xi}\}_{\xi\in \pi_r(A)}$ with $f_{\xi}: A_{\xi}\to K^l$ is semi-algebraic if the graph of $f: A\to K^l: (\xi,x)\mapsto (\xi, f_{\xi}(x))$ is a semi-algebraic set.\\
We say that a family of semi-algebraic functions $f =\{f_{\xi}\}_{\xi\in \pi_r(A)}$ with $f_{\xi}: A_{\xi}\to K$ 
satisfies condition (\ref{paramconditionf}) if there exist constants $\mu_i\in\Z$ and a semi-algebraic function $\beta: \pi_r(A) \to K$ such
that each $(\xi,x_1,\ldots, x_m)\in A$ satisfies $x_i\not=0$ if $\mu_i<0$ and 
\begin{equation}\label{paramconditionf}
 \ord(f_{\xi}(x))=\ord(\beta(\xi)\prod_i x_i^{\mu_i}).
 \end{equation}
We say that a family of semi-algebraic functions
$g =\{g_{\xi}\}_{\xi\in \pi_r(A)}$ with $g_{\xi}: A_{\xi}\to K^m$
 satisfies condition (\ref{paramconditionjac})
 if  each $g_{\xi}$ is $C^1$ on $A_{\xi}$ and there
exist constants $\mu_i\in\Z$ and a semi-algebraic function $\beta: \pi_{r}(A)\to K$ such that each
$(\xi, x_1, \ldots, x_m)\in A$ satisfies $x_i\not=0$ if $\mu_i<0$ and 
\begin{equation}\label{paramconditionjac}
 \ord(\text{Jac}(g_{\xi}(x)))=\ord(\beta(\xi)\prod_i x_i^{\mu_i}).
 \end{equation}
 }
 \end{definition}

 \begin{definition}
Let $f:X\to Y$ be a family of semi-algebraic isomorphisms. Say that the
family $f$ is of type $f_0$ when $f$ equals an isomorphism of
the following kind:
\begin{eqnarray*}
f_0 &:&\hspace{-2pt} X\subset K^{r+n} \to Y \subset K^{r+n+1}: (\xi,x)
\mapsto (\xi, x_1, \ldots, x_{i-1},0,x_{i+1},\ldots, x_{n})\\
\end{eqnarray*}
for some $n\geq 0$.\\
Say that the family $f$ is of type $f_1$, resp. of type
$f_2$ or  $t_c$, when ${X \subset K^{r+n}}$,  ${Y\subset K^{r+n}}$ for
some $n\geq 0$, and $f$ equals an isomorphism of the following
kind:
\begin{eqnarray*}
f_1 \hspace{-5pt}&:&\hspace{-5pt} X \to Y: (\xi,x) \mapsto(\xi, \alpha_1 x_1^{a_1}, \ldots,
\alpha_n x_n^{a_n}),\\
f_2\hspace{-5pt}&:& \hspace{-5pt}X \to Y: (\xi,x) \mapsto (\xi, x_1, \ldots, x_{i-1}, x_i
\beta(\xi)\prod_{j\neq i}x_j^{b_j}, x_{i+1}, \ldots, x_n),\mbox{ or}\\
t_c\hspace{-5pt}&:&\hspace{-5pt} X\to Y : (\xi,x) \mapsto(\xi, x_1, \ldots, x_{i-1}, x_i + c(\xi, x_1,
\ldots, x_{i-1}), x_{i+1}, \ldots, x_n), \\
\end{eqnarray*}
with $a_i, b_j \in \Z$, $a_i<0$ implies $x_i\not = 0 $ for $x\in X$,
$b_j<0$ implies $x_j\not = 0 $ for $x\in X$, $\alpha_i \in K, \
{x = (x_1, \ldots, x_n)}$, $\beta: \pi_{r}(X)\to K$ a semi-algebraic function and $c:\pi_{r+i-1}(X) \to K$ a semialgebraic function which is
moreover $C^1$ in the case that $X$ is open in $K^{r+n}$.

Given a function $f: (\xi,x)\mapsto (\xi, f_{\xi}(x))$ of one of the types described above, we will also say that the functions $f_{\xi}$ have type $f_{i}$ or $t_{c}$.
\end{definition}

\begin{theorem}[Parametric Rectilinearization with basic functions]\label{thm:param-rect}
 Let $X \subset K^{r+ m}$ be a semi-algebraic set and $b_j:X\to K$ semi-algebraic functions for
$j=1,\ldots,n$. Then there exists a finite partition of $X$ such
that for each part $A= \cup_{\xi \in \pi_r(A)} A_{\xi}$ we have constants $l\in\N$, $k\in\N_0$,
 and a family of  isomorphisms $f = \{f_{\xi}\}_{\xi\in \pi_r(A)}$ with
\[f_{\xi}:\Delta_{\xi} \times \prod_{i=l'}^l R^{(k)}\to A_{\xi
}: x \mapsto (T_c \circ g)(x),\]
where 
 the values $l,l' \leqslant m$ only depend on the part $A$.
$\Delta_{\xi}$ is a set of the form $\{x \in \prod_{i=1}^{l'-1} R^{(k)} \mid (\ord x_1, \ldots, \ord x_{l'-1}) \in \Gamma_A\} $. For each $\xi \in \pi_r(A)$,  $\Gamma_A$ is a finite subset of $\Gamma_{K}^{l'-1}$ containing elements $\gamma = (\gamma_1, \ldots, \gamma_{l'-1})$ that satisfy relations
\[0 \leqslant \gamma_j \leqslant \ord (\beta_j(\xi)) + \sum_{i=1}^{j-1}n_i\gamma_i\]
for nonzero semi-algebraic functions $\beta_i: \pi_r(A)\to K$ and $n_i,m_i \in \Z$ for $i =1,\ldots, l'-1$. 
The function $g$ is a composition of bijective maps of the types
$f_0,f_1,f_2$, and $T_c$ a composition of functions of type
$t_c$.
\\\\
This implies that the functions $b_j \circ f$ satisfy condition
(\ref{paramconditionf}). Moreover, for each part $A$ where
 $l = m$, $f$ satisfies condition
(\ref{paramconditionjac}).
\end{theorem}
\noindent We will need the following lemma: 
\begin{lemma} \label{lemma:param-Eisom}
Let $S\subseteq K^{r}$ be a semi-algebraic set, and $E = \cup_{\xi \in S} E_{\xi} \subseteq K^m$ be a family of sets, with $E_{\xi}$ of the form
\[E_{\xi}:=\left\{x\in \Delta_{\xi} \times \prod_{i=m'}^{m}R^{(k)} \mid \ord x_m\ \leq
 \ \ord\left(\beta(\xi)\prod_{i=1}^{m-1} x_i^{\nu_i}\right)\right\},\]
where $\beta(\xi)$ is a nonzero semi-algebraic function, $k \in \N_0$, $\nu_i \in \Z$, $\Delta_{\xi}$ is a bounded subset of $\prod_{i=1}^{m'} R^{(k)}$, $m, m' \in \N$ depending only on $E$.\\ There exists a finite partition  of  $S$ in semi-algebraic sets $S_i$ and for each $E_{S_i} = \cup_{\xi \in S_i} E_{\xi}$, a finite partition in parts $A$ such
that for each part $A = \cup_{\xi \in S_i} A_{\xi}$,  there is a family of isomorphisms $f_{\xi}$ of the form
 \[f_{\xi}: \Sigma_{\xi} \times \prod_{i=l'}^l R^{(k)} \to A_\xi: x \mapsto (g \circ F_0)(x).\] 
Here $\Sigma_{\xi}$ is  a bounded set, $l', l \leq  m$ are independent of $\xi$,  $g$ is a composition of functions of types $f_1, f_2$,
and $F_0$ is a composition of functions of type $f_0$.
\end{lemma}

\begin{proof}
Partition $S$ in $S_1 = \{\xi \in S \mid \beta(\xi) = 0\}$ and $S_2 = S \backslash S_1$. The Lemma is trivial for $E_{S_1}$, so we may assume that $\beta(\xi)$ is nonzero.

We work by induction on $m$. The case where $m'=m$ is trivial, so assume that $m > m'$.  
Note that for $m' \leqslant i <m$, there are no conditions on $x_i$ if
$\nu_i =0$. 

We first look at the case where $\nu_i < 0$ for all $i=m', \ldots, m-1$. In this case each $E_{\xi}$ is a bounded set. Indeed, we have that
\begin{align*} 0\leqslant \ord x_m&\leqslant \ord \beta(\xi) + \sum_{i=1}^{m'-1} \nu_i\ord x_i - \sum_{i=m'}^{m-1}|\vu_i| \ord x_i,\\
&\leqslant M(\xi) - \sum_{i=m'}^{m-1}|\vu_i| \ord x_i,
\end{align*}
where $M(\xi)$ is a natural number that may depend on $\xi.$
This implies that $E_{\xi}$ cannot contain any $x$ for which $\ord x_i > \frac{ \ord M(\xi)}{|\nu_i|}$ for some $m' \leqslant i <m$.

Hence, we may suppose that $\nu_{m-1}>0$.
 We first prove the proposition when $\nu_{m-1}=1$. We can partition $E_{\xi}$
into parts $E_{\xi,1}$ and $E_{\xi,2}$, with
\begin{eqnarray*}
 \hspace{12 mm}E_{\xi,1} & = & \left\{x\in
 E_{\xi}\mid  \ord x_m \leq \ord \left(\beta(\xi)\prod_{i=1}^{m-2}x_i^{\nu_i}\right)\right\},\\
 E_{\xi,2} & = & \left\{x\in
 E_{\xi}\mid  \ord \left(\beta(\xi)\prod_{i=1}^{m-2}x_i^{\nu_i}\right)< \ord x_m\right\}\\
 & = & \left\{x\in\Delta_{\xi} \times \prod_{i=m'}^m R^{( k)} \mid
 \ord\left(\beta(\xi)\prod_{i=1}^{m-2}x_i^{\nu_i}\right) <
 \ord x_m\right. \\&& \hspace{130pt}\left.\leq \ord\left(\beta(\xi) x_{m-1}\prod_{i=1}^{m-2}x_i^{\nu_i}\right)\right\}.
\end{eqnarray*}
Since $ \ord(\beta(\xi)\prod_{i=1}^{m-2}x_i^{\nu_i})\leq
 \ord(x_{m-1}\beta(\xi)\prod_{i=1}^{m-2}x_i^{\nu_i})$ for $x\in E_{\xi,1}$, it follows that we have an isomorphism \[E'_{\xi,1}\to E_{\xi,1}:(x_1,\ldots, x_{m-2}, x_{m-1} ,x_m) \mapsto(x_1,\ldots, x_{m-2},x_{m},x_{m-1})\] (which is a composition of maps of type $f_1$ and $f_{2}$) with $E'_{\xi,1}$ the set
 \[
 \left\{(x_1,\ldots,x_{m-1})\in \Delta_{\xi} \times \prod_{i=m'}^{m-1}R^{(k)} \mid
 \ord(x_{m-1})\leq
 \ord\hspace{-3pt}\left(\hspace{-2pt}\beta(\xi)\hspace{-2pt}\prod_{i=1}^{m-2}\hspace{-2pt}x_i^{\nu_i}\hspace{-2pt}\right)\right\}\times R^{(k)}\hspace{-2pt},
 \]
and the lemma follows for $E'_{\xi,1}$ by the induction hypothesis.\\
 For $E_{\xi,2}$, let $D_{\xi,m-1}$ be the set
\[\left\{(x_1,\ldots,x_{m-1})\in\Delta_{\xi}\times\prod_{i=m'}^{m-1}R^{(k)}\mid \ord
 \left(\beta(\xi)\prod_{i=1}^{m-2}x_i^{\nu_i}\right) < \ord x_{m-1}\right\}.
\]
We may suppose that $\beta(\xi)\in K^{( k)}$. Then, the map
\[
 D_{\xi,m-1}\times R^{(k)} \to E_{\xi,2}: x\mapsto\left(x_1, \ldots, x_{m-2}, \frac{x_{m-1}x_m}{
\beta(\xi)\prod_{i=1}^{m-2}x_i^{\nu_i}},x_{m-1}\right)
\]
is an isomorphism which is a composition of isomorphisms of type
$f_1$ and $f_2$. Also
$$
\Delta_{\xi} \times \prod_{i=m'}^m\hspace{-3pt} R^{( k)}\to  D_{\xi,m-1}\times R^{(k)}: x \mapsto\hspace{-2pt}
\left(\hspace{-2pt}x_1, \ldots, x_{m-2}, \pi \beta(\xi)\, x_{m-1}
\prod_{i=1}^{m-2}\hspace{-3pt}x_i^{\nu_i}, x_m \hspace{-2pt}\right)
$$
is an isomorphism which is a composition of isomorphisms of type
$f_1$ and $f_2$. This proves the lemma when $\nu_{m-1}=1$.
\par
Suppose now that $\nu_{m-1}>1$. We prove that we can reduce to the case
$\nu_{m-1}=1$ by partitioning and applying appropriate power maps.
Choose $\tilde k>\ord(\nu_{m-1})$ and put $\tilde k'=\tilde k+\ord(\nu_{m-1})$. We
may suppose that $\tilde k\geq k$, so we have a finite partition
$E_{\xi}=\bigcup_\alpha E_{\xi,\alpha}$, with
${\alpha=(\alpha_1,\ldots,\alpha_m)\in K^m}$, $\ord(\alpha_1)=0$,
$0\leq \ord(\alpha_i)<\nu_{m-1}$ for $i=2,\ldots,m$ and
\begin{align*}E_{\xi,\alpha}=\{x\in E_{\xi}\mid  x_{m-1}\in \alpha_{m-1}R^{(\tilde k)}, &\ x_i\in \alpha_i
P_{\nu_1}^{(\tilde k')} \\&\mbox{ for } i\in\{1,\ldots,m-2\}\cup \{m\}\}.\end{align*} By Corollary
\ref{corhensel} we have isomorphisms
\[f_{\xi,\alpha}:C_{\xi,\alpha}\to E_{\xi,\alpha}:x\mapsto(\alpha_1x_1^{\nu_{m-1}},\ldots , \alpha_{m-2}x_{m-2}^{\nu_{m-1}},\alpha_{m-1}x_{m-1},
\alpha_mx_m^{\nu_{m-1}}),\]
 with $C_{\xi,\alpha}=\{x\in\Sigma_{\xi} \times \prod_{i=m'}^m R^{(\tilde k)}\mid \ord(x_m)\leq \ord
(\beta' x_1\prod_{i=2}^{m-1}x_i^{\nu_i})\}$, which are isomorphisms
of type $f_1$.  Here $\Sigma_{\xi}$ is a bounded (definable) subset of $\prod_{i=1}^{m'-1} R^{(\tilde k)}$ and $\beta': S \to K^{\times}$ a semi-algebraic function. This reduces the problem to the case with $\nu_{m-1}=1$ and
thus the lemma is proved.
\end{proof}

\begin{proof}[Proof of Theorem \ref{thm:param-rect}.]

We give a proof by induction on $m$. We will show that we can reduce to the case described
in equation (\ref{eq paramdim>1}) below. The theorem then follows by
Lemma \ref{lemma:param-Eisom}. Using Lemma \ref{partition} and its
notation, we find a finite partition of $X$ such that each part $A$
has the form
\begin{eqnarray*}
A\hspace{-3pt} & = &\hspace{-3pt} \{(\xi,x,t)\in K^{r+(m+1)} \mid   \ord a_1(\xi,x)\sq_1 \ord(t-c(\xi,x))\sq_2 \ord a_2(\xi,x),\\
 & &\hspace{100pt} (\xi,x) \in D,\ t-c(\xi,x)\in \lambda P_n\},
\end{eqnarray*}
 and such that  $ \ord b_j(\xi,x,t)=\frac{1}{n}
\ord((t-c(\xi,x))^{\mu_{mj}}d_j(\xi,x))$ for each $x\in A$, with $\mu_{mj}\in\Z$.

First use a translation $t_c^{(m+1)}: A_{c=0} \to A: x\mapsto
(x_1,\ldots, x_m, t+c(\xi,x))$. If we apply the induction hypothesis to the set
$D\subset K^{r+m}$ and the functions $a_1,a_2,d_j$, we get a
partition of $D$ in parts $\tilde D$, such that for each $\tilde D$,
 there exists a family of definable functions $\{\tilde f_\xi\}_{\xi \in \pi_m(\tilde D)}$ so that we have isomorphisms of the form
  \[\tilde f_\xi:
\Delta_{\xi} \times \prod_{i=l'}^l R^{(k)}\to \tilde D_\xi: (x_1, \ldots, x_l) \mapsto (\tilde
T_c \circ \tilde g)(x_1, \ldots, x_l),\] with $\Delta_{\xi}$ a bounded definable set (of the form described in the formulation of the theorem) and $l,l'\in \N$ depending only on $\tilde D$.   $\tilde T_c$ is a
composition of functions of type $t_c$ and $\tilde g$ a composition
of functions of types $f_0,f_1,f_2$. This induces a finite
partition of $A_{c=0}$ in parts $\tilde A_{c=0}$, such that for each part
there is an isomorphism of the form
\[{\tilde f':B\to \tilde A_{c=0}: (\xi,x,t) \mapsto (\xi,\tilde f_\xi(x_1,\ldots,x_l), t ).}\]
For each $\xi \in \pi_m(B)$, $B_{\xi}$ is a set of the form $ B_{\xi}=\{ (x,t) \in  \Delta_{\xi} \times \prod_{i=l'}^lR^{(k)}\times \lambda P_n\mid \phi(\xi,x,t)\}$ with
\[\phi(\xi,x,t) \leftrightarrow 
   \ord\left(\alpha_1(\xi)\prod_{i=1}^lx_i^{\eta_i}\right)
\sq_1 \ \ord t\ \sq_2
\ord\left(\alpha_2(\xi)\prod_{i=1}^lx_i^{\varepsilon_i}\right),
\]
where the $\alpha_i: \pi_m(B)\to K$ are semi-algebraic functions.
If we compose the functions $\tilde f' $ with the translation $t_c^{(m+1)}$, we obtain isomorphisms of the form
\[f = t_c^{(m+1)} \circ \tilde f':B \to A_B: (\xi,x,t) \mapsto ((\tilde T_c \circ \tilde g)(\xi,x),
 t +c(\xi,x)) \]
between sets $B$ and sets $A_B \subset X$. The sets $A_B$ form a
finite partition of $X$. Applying the induction hypotheses, we find that there exist $\mu_{ij}\in \Z$ and  semi-algebraic functions $\delta_i: \pi_r(B)\to K$ such that $\ord (b_j \circ f)(\xi,x,t)=$ $ \frac1n\ord(\delta_i(\xi)t^{\mu_{l+1,j}}\prod_{k=1}^l x_i^{\mu_{kj}})$. 

We will show that we only need functions of type $f_0, f_1$ and $ f_2$  for our parametric rectilinearization of $B$.
 Thus the final isomorphisms $\prod R^{(k)} \to A$ will have the form $x \mapsto (T_c \, \circ \, g)(x)$, where $g$ is a composition of functions of types $f_0, f_1$ and $f_2$, and
 $T_c$ is a composition of functions of type $t_c$.
\\\\
If $\lambda=0$, then $B_{\xi} =\Delta_{\xi} \times \prod_{i=l'}^{l}R^{(k)}\times\{0\}$. In
this case our isomorphism has the form
$(\xi,x) \mapsto (T_c  \circ g \circ
f_0)(\xi,x).$ Recall that by Lemma \ref{partition}, $\alpha_1(\xi) \neq0
\neq\alpha_2(\xi)$. From now on suppose that  $\lambda \neq 0$.

As in the nonparametric case, we may suppose that $\sq_2$ is either
$\leq$ or no condition and $\sq_1$ is the symbol $\leq$ (possibly
after partitioning or applying
$(\xi,x,t)\mapsto(\xi,x,1/t)$).
\\

Choose $\bar k>\ord(n)$ and put $k'=\bar k+\ord(n)$. We may suppose that
$k'>k$, so we have a finite partition $B=\bigcup_\gamma B_\gamma$
with $\gamma=(\gamma_1,\ldots,\gamma_{l+1})\in K^{l+1}$, $0\leq
\ord(\gamma_i)<n$ and
\[B_\gamma=\{(\xi,x,t) \in B \mid t\in \gamma_{l+1}P_n^{(k')}, x_i\in\gamma_i P_{n}^{(k')},\text{ for } i=1,\ldots, l\}.\] Now we have isomorphisms
\[
f_\gamma:C_\gamma\to
B_\gamma:(\xi,x,t) \mapsto(\xi, \gamma_1x_1^{n},\ldots,\gamma_{l}x_{l}^{n}, \gamma_{l+1}t^n).
\]
 For each $\xi \in \pi_m(B_{\gamma})$, the set $C_{\gamma,\xi}$ is defined as $\{x\in \Delta_\xi' \times\prod_{i=l'}^{l}R^{(\bar k)}\times
K^{(\bar k)}\mid \psi(\xi,x,t)\}$, with
\[
\psi(\xi,x,t) \leftrightarrow  \ord\left(\alpha'_1(\xi)\prod_{i=1}^{l}x_i^{\eta_i}\right)
\leq \ord t\ \sq_2\
\ord\left(\alpha'_2(\xi)\prod_{i=1}^{l}x_i^{\varepsilon_i}\right),
\]
where the $\alpha_i': \pi_m(C_{\gamma}) \to K$ are semi-algebraic functions (for this we need Lemma \cite{denef-86}.2.4)  Put $\overline f = f
\circ f_{\gamma}$. Then the $b_{j}\circ\overline f$ satisfy
condition (\ref{paramconditionf}), since
\begin{eqnarray*}
\ord(b_j \circ f(f_\gamma(\xi,x,t))) &=& \frac1n \ord \left(\beta_j(\xi)t^{\mu_{l+1,j}n}\prod_{i=1}^{l+1} \gamma_i^{\mu_j}\prod_{i=1}^l x_i^{\mu_{ij}n}\right)\\
&=&\frac1n \ord \left(\beta_j(\xi) \prod_{i=1}^{l+1} \gamma_i^{\mu_j}\right)+ \ord\left(t^{\mu_{l+1,j}}\prod_{i=1}^l x_i^{\mu_{ij}}\right),\end{eqnarray*}
and by Lemma \cite{denef-86}.2.4 there exists a semi-algebraic function $\tilde{\beta}_j(\xi)$ such that $\ord\tilde{\beta}_j(\xi) = \frac1n \ord \left(\beta_j(\xi) \prod_{i=1}^{l+1} \gamma_i^{\mu_j}\right)$.

Put $\nu_i=\varepsilon_i-\eta_i$,
$\beta(\xi)=\alpha'_2(\xi)/\alpha'_1(\xi)$. Then the following is an isomorphism
\[
D_{\gamma} \to C_\gamma:
 (\xi,x,t)\mapsto\left(\xi, x_1,\ldots,x_{l},\alpha'_1(\xi)t\displaystyle\prod_{i=1}^{l}x_i^{\eta_i}\right),
\] with  $D_{\gamma,\xi} = \left\{x\in \Delta_{\xi} \times \displaystyle\prod_{i=1}^{l+1}R^{(\bar k)}\mid \ \ord t\sq_2
 \ord\left(\beta(\xi)\prod_{i=1}^{l} x_i^{\nu_i}\right)\right\}$.
\par
The case that $\sq_2$ is no condition is now trivial. Summarizing,
it follows that we can reduce to the case of an isomorphism $f:E\to X: (\xi, x,t) \mapsto (\xi, f_\xi(x,t))$, with $f_{\xi}$ a map $f_{\xi}:E_{\xi} \to X_\xi$, where $E_{\xi}$ is the set
\begin{equation}\label{eq paramdim>1}
E_{\xi} = \left\{(x,t) \in \Delta_{\xi} \times \prod_{i=l'}^{l}R^{(\bar k)}\mid \ord t \ \leq
 \ord \left(\beta(\xi)\prod_{i=1}^{l} x_i^{\nu_i}\right)\right\},
\end{equation}
with $\beta(x)$ a nonzero semi-algebraic function, $\bar k>0$, and $\nu_i\in\Z$, such that each
$b_j\circ f$  satisfies condition (\ref{paramconditionf}). 
\par
Use Lemma \ref{lemma:param-Eisom} to obtain a partition of $E$ in parts
$E_i =\cup_{\xi}E_{\xi,i}$ and families of isomorphisms $\phi_{\xi,i}: \Delta_{\xi} \times \prod R^{(k)} \to E_i$.  The
$\phi_{\xi,i}$ are composed of functions of types $f_0, f_1, f_2$ and
the components of $\phi_{\xi,i}$ all satisfy condition
(\ref{paramconditionf}). Therefore each $b_j \circ f \circ \phi_i$ will
satisfy condition $(\ref{paramconditionf})$. That the condition on the Jacobians holds can now be checked in a straightforward way (for a proof, see \cite{clu-lee-2008}). 
Our claim on the form of $\Delta_{\xi}$ follows immediately from the proof of Lemma \ref{lemma:param-Eisom}.
This finishes the proof of
Theorem \ref{thm:param-rect}.
\end{proof}

\begin{lemma} \label{lemma:param-Eisom}
Let $S\subseteq K^{r}$ be a semi-algebraic set, and $E = \cup_{\xi \in S} E_{\xi} \subseteq K^m$ be a family of sets, with $E_{\xi}$ of the form
\[E_{\xi}:=\left\{x\in \Delta_{\xi} \times \prod_{i=m'}^{m}R^{(k)} \mid \ord x_m\ \leq
 \ \ord\left(\beta(\xi)\prod_{i=1}^{m-1} x_i^{\nu_i}\right)\right\},\]
where $\beta(\xi)$ is a nonzero semi-algebraic function, $k \in \N_0$, $\nu_i \in \Z$, $\Delta_{\xi}$ is a bounded subset of $\prod_{i=1}^{m'} R^{(k)}$, $m, m' \in \N$ depending only on $E$.\\ There exists a finite partition  of  $S$ in semi-algebraic sets $S_i$ and for each $E_{S_i} = \cup_{\xi \in S_i} E_{\xi}$, a finite partition in parts $A$ such
that for each part $A = \cup_{\xi \in S_i} A_{\xi}$,  there is a family of isomorphisms $f_{\xi}$ of the form
 \[f_{\xi}: \Sigma_{\xi} \times \prod_{i=l'}^l R^{(k)} \to A_\xi: x \mapsto (g \circ F_0)(x).\] 
Here $\Sigma_{\xi}$ is  a bounded set, $l', l \leq  m$ are independent of $\xi$,  $g$ is a composition of functions of types $f_1, f_2$,
and $F_0$ is a composition of functions of type $f_0$.
\end{lemma}

\begin{proof}
Partition $S$ in $S_1 = \{\xi \in S \mid \beta(\xi) = 0\}$ and $S_2 = S \backslash S_1$. The Lemma is trivial for $E_{S_1}$, so we may assume that $\beta(\xi)$ is nonzero.

We work by induction on $m$. The case where $m'=m$ is trivial, so assume that $m > m'$.  
Note that for $m' \leqslant i <m$, there are no conditions on $x_i$ if
$\nu_i =0$. 

We first look at the case where $\nu_i < 0$ for all $i=m', \ldots, m-1$. In this case each $E_{\xi}$ is a bounded set. Indeed, we have that
\begin{align*} 0\leqslant \ord x_m&\leqslant \ord \beta(\xi) + \sum_{i=1}^{m'-1} \nu_i\ord x_i - \sum_{i=m'}^{m-1}|\vu_i| \ord x_i,\\
&\leqslant M(\xi) - \sum_{i=m'}^{m-1}|\vu_i| \ord x_i,
\end{align*}
where $M(\xi)$ is a natural number that may depend on $\xi.$
This implies that $E_{\xi}$ cannot contain any $x$ for which $\ord x_i > \frac{ \ord M(\xi)}{|\nu_i|}$ for some $m' \leqslant i <m$.

Hence, we may suppose that $\nu_{m-1}>0$.
 We first prove the proposition when $\nu_{m-1}=1$. We can partition $E_{\xi}$
into parts $E_{\xi,1}$ and $E_{\xi,2}$, with
\begin{eqnarray*}
 \hspace{12 mm}E_{\xi,1} & = & \left\{x\in
 E_{\xi}\mid  \ord x_m \leq \ord \left(\beta(\xi)\prod_{i=1}^{m-2}x_i^{\nu_i}\right)\right\},\\
 E_{\xi,2} & = & \left\{x\in
 E_{\xi}\mid  \ord \left(\beta(\xi)\prod_{i=1}^{m-2}x_i^{\nu_i}\right)< \ord x_m\right\}\\
 & = & \left\{x\in\Delta_{\xi} \times \prod_{i=m'}^m R^{( k)} \mid
 \ord\left(\beta(\xi)\prod_{i=1}^{m-2}x_i^{\nu_i}\right) <
 \ord x_m\right. \\&& \hspace{130pt}\left.\leq \ord\left(\beta(\xi) x_{m-1}\prod_{i=1}^{m-2}x_i^{\nu_i}\right)\right\}.
\end{eqnarray*}
Since $ \ord(\beta(\xi)\prod_{i=1}^{m-2}x_i^{\nu_i})\leq
 \ord(x_{m-1}\beta(\xi)\prod_{i=1}^{m-2}x_i^{\nu_i})$ for $x\in E_{\xi,1}$, it follows that we have an isomorphism \[E'_{\xi,1}\to E_{\xi,1}:(x_1,\ldots, x_{m-2}, x_{m-1} ,x_m) \mapsto(x_1,\ldots, x_{m-2},x_{m},x_{m-1})\] (which is a composition of maps of type $f_1$ and $f_{2}$) with $E'_{\xi,1}$ the set
 \[
 \left\{(x_1,\ldots,x_{m-1})\in \Delta_{\xi} \times \prod_{i=m'}^{m-1}R^{(k)} \mid
 \ord(x_{m-1})\leq
 \ord\hspace{-3pt}\left(\hspace{-2pt}\beta(\xi)\hspace{-2pt}\prod_{i=1}^{m-2}\hspace{-2pt}x_i^{\nu_i}\hspace{-2pt}\right)\right\}\times R^{(k)}\hspace{-2pt},
 \]
and the lemma follows for $E'_{\xi,1}$ by the induction hypothesis.\\
 For $E_{\xi,2}$, let $D_{\xi,m-1}$ be the set
\[\left\{(x_1,\ldots,x_{m-1})\in\Delta_{\xi}\times\prod_{i=m'}^{m-1}R^{(k)}\mid \ord
 \left(\beta(\xi)\prod_{i=1}^{m-2}x_i^{\nu_i}\right) < \ord x_{m-1}\right\}.
\]
We may suppose that $\beta(\xi)\in K^{( k)}$. Then, the map
\[
 D_{\xi,m-1}\times R^{(k)} \to E_{\xi,2}: x\mapsto\left(x_1, \ldots, x_{m-2}, \frac{x_{m-1}x_m}{
\beta(\xi)\prod_{i=1}^{m-2}x_i^{\nu_i}},x_{m-1}\right)
\]
is an isomorphism which is a composition of isomorphisms of type
$f_1$ and $f_2$. Also
$$
\Delta_{\xi} \times \prod_{i=m'}^m\hspace{-3pt} R^{( k)}\to  D_{\xi,m-1}\times R^{(k)}: x \mapsto\hspace{-2pt}
\left(\hspace{-2pt}x_1, \ldots, x_{m-2}, \pi \beta(\xi)\, x_{m-1}
\prod_{i=1}^{m-2}\hspace{-3pt}x_i^{\nu_i}, x_m \hspace{-2pt}\right)
$$
is an isomorphism which is a composition of isomorphisms of type
$f_1$ and $f_2$. This proves the lemma when $\nu_{m-1}=1$.
\par
Suppose now that $\nu_{m-1}>1$. We prove that we can reduce to the case
$\nu_{m-1}=1$ by partitioning and applying appropriate power maps.
Choose $\tilde k>\ord(\nu_{m-1})$ and put $\tilde k'=\tilde k+\ord(\nu_{m-1})$. We
may suppose that $\tilde k\geq k$, so we have a finite partition
$E_{\xi}=\bigcup_\alpha E_{\xi,\alpha}$, with
${\alpha=(\alpha_1,\ldots,\alpha_m)\in K^m}$, $\ord(\alpha_1)=0$,
$0\leq \ord(\alpha_i)<\nu_{m-1}$ for $i=2,\ldots,m$ and
\begin{align*}E_{\xi,\alpha}=\{x\in E_{\xi}\mid  x_{m-1}\in \alpha_{m-1}R^{(\tilde k)}, &\ x_i\in \alpha_i
P_{\nu_1}^{(\tilde k')} \\&\mbox{ for } i\in\{1,\ldots,m-2\}\cup \{m\}\}.\end{align*} By Corollary
\ref{corhensel} we have isomorphisms
\[f_{\xi,\alpha}:C_{\xi,\alpha}\to E_{\xi,\alpha}:x\mapsto(\alpha_1x_1^{\nu_{m-1}},\ldots , \alpha_{m-2}x_{m-2}^{\nu_{m-1}},\alpha_{m-1}x_{m-1},
\alpha_mx_m^{\nu_{m-1}}),\]
 with $C_{\xi,\alpha}=\{x\in\Sigma_{\xi} \times \prod_{i=m'}^m R^{(\tilde k)}\mid \ord(x_m)\leq \ord
(\beta' x_1\prod_{i=2}^{m-1}x_i^{\nu_i})\}$, which are isomorphisms
of type $f_1$.  Here $\Sigma_{\xi}$ is a bounded (definable) subset of $\prod_{i=1}^{m'-1} R^{(\tilde k)}$ and $\beta': S \to K^{\times}$ a semi-algebraic function. This reduces the problem to the case with $\nu_{m-1}=1$ and
thus the lemma is proved.

\end{proof}
\subsection{Application to $p$-adic integration}
Using parametric rectilinearization, we can give a new proof of the rationality of $p$-adic integrals with parameters. This was originally proven by  Denef, see \cite{denef-85, denef-2000}.
The proof uses the rationality result for p-adic integrals due to Denef, of which we provided an alternative proof in \cite{clu-lee-2008}.

\begin{theorem}[Rationality, \cite{denef-84}]\label{crat}
Let $S \subset K^m$ be a semi-algebraic set and $  f,\, g: S \to K$
semi-algebraic functions. If the following integral exists for $s
\in \R, s
>>0$ (that is, if the integrand is absolutely integrable for $s$ sufficiently big),  then
 \begin{equation}\label{int}
I(s) := \int_{S} |f(x)|^s\cdot|g(x)| 
|dx|
 \end{equation}
 is rational in $q_K^{-s}$ and the denominator of $I(s)$ is a
product of factors of the form $(1-q_K^{-sa-b})$ with $a,b \in
\Z$, and $(a,b) \neq (0,0)$.
\end{theorem}

\begin{proposition}
Let $S \subseteq K^{r+m}$ be a semi-algebraic set and $f,g:S \to K$ a semi-algebraic function. For every $\xi \in \pi_r{S}$, let $I_{\xi}(s)$ be the following integral
\[I_{\xi}(s) := \int_{S_{\xi}} | f(\xi,x)|^s|g(\xi,x)| |dx|.\]
If $I_{\xi}(s)$ exists for $s >>0$, then the integral is rational in $q_{K}^{-s}$. More precisely,
\[ I_{\xi}(s) = |\beta(\xi)|^s|\gamma(\xi)| \frac{P_{\xi}(q_K^{-s})}{Q(q_K^{-s})},\]
where $\beta, \gamma$ are semi-algebraic functions $\pi_r(S) \to K$, and $Q(T), P_{\xi}(T) \in K[T]$.
 $Q(T)$ is a product of factors $(1-q_K^{-b}T^a)$ with $a,b \in \Z$. These factors do not depend on $\xi$. 
The degree of $P_{\xi}(T)$ is bounded by the order of a semi-algebraic function in the variables $\xi$.
\end{proposition}
\begin{proof}
By Theorem \ref{thm:param-rect} and the change of variables formula for $p$-adic integrals,  $I_{\xi}(s)$ is equal to a finite linear combination (with constant coefficients) of integrals
\[|\beta(\xi)|^s |\gamma(\xi)|\int_{\Delta_{\xi}\times \prod_{i=l'}^{l}R^{(k)}} \left|\prod_{i=1}^lx_i^{\mu_i}\right|^s \left| \prod_{i=1}^l x_i^{\nu_i}\right| |dx|.\]
By Theorem \ref{crat}, this is equal to
\[R(q_K^{-s})\cdot I'_{\xi}(s):=R(q_K^{-s})\cdot|\beta(\xi)|^s| \gamma(\xi)|\int_{\Delta_{\xi}} \left|\prod_{i=1}^{l'-1}x_i^{\mu_i}\right|^s \left| \prod_{i=1}^{l'-1} x_i^{\nu_i}\right| |dx|,\]
where $R(T)\in K(T)$ and the denominator consists of factors  $(1-q_K^{-b}T^{a})$ with $a,b \in \Z$. We know by Theorem \ref{thm:param-rect} that $\Delta_{\xi}$ is a set of the form \[\left\{x \in \prod_{i=1}^{l'-1}R^{(k)}\mid (\ord x_1, \ldots, \ord x_{l'-1}) \in \Gamma_{S}\right\} \] with $\Gamma_{S}$ a finite subset of $\Gamma_{K}^{l'-1}$ containing elements $\alpha = (\alpha_1, \ldots, \alpha_{l'-1})$ that satisfy relations
\[0 \leqslant \alpha_j \leqslant \ord (\beta_j(\xi)) + \sum_{i=1}^{j-1}n_i\alpha_i\]
for non-zero semi-algebraic functions $\beta_i: \pi_r(S)\to K$ and $n_i,m_i \in \Z$ for $i =1,\ldots, l'-1$. 
Because of this, $I'_{\xi}(s)$ is equal to 
\[|\beta(\xi)|^s|\gamma(\xi)| q_K^{(l'-1)k}\sum_{\alpha \in \Gamma_S} q_K^{-\sum_{i=1}^{l'-1} (\mu_is +\nu_i-1)\alpha_i}. \]
As this is a finite sum (with the exact number of terms depending on $\xi$), this proves our claim.
\end{proof}

\subsection*{Acknowledgements}
I would like to thank Raf Cluckers for suggesting this topic to me.

\bibliographystyle{hplain}
\bibliography{/Users/iblueberry/Documents/Bibliografie}

\end{document}